\documentclass[12pt]{amsart}
\usepackage{amsmath}
\usepackage{amssymb}
\usepackage{fancybox}
\usepackage{eucal}
\usepackage{amsthm}
\usepackage{amscd}
\usepackage{hyperref}
\usepackage{wasysym}
\usepackage{url}

\usepackage{amssymb,}
\usepackage[]{amsmath, amsthm, amsfonts,graphicx, amscd,}
\usepackage[all,cmtip]{xy}
\input amssym.def \input amssym

\begin{document}

\newtheorem {thm}{Theorem}[section]
\newtheorem{corr}[thm]{Corollary}
\newtheorem {cl}[thm]{Claim}
\newtheorem*{thmstar}{Theorem}
\newtheorem{prop}[thm]{Proposition}
\newtheorem*{propstar}{Proposition}
\newtheorem {lem}[thm]{Lemma}
\newtheorem {fact}[thm]{Fact}
\newtheorem*{lemstar}{Lemma}
\newtheorem{conj}[thm]{Conjecture}
\newtheorem{question}[thm]{Question}
\newtheorem*{questar}{Question}
\newtheorem{ques}[thm]{Question}
\newtheorem{prob}[thm]{Problem}
\newtheorem*{conjstar}{Conjecture}
\theoremstyle{remark}
\newtheorem{rem}[thm]{Remark}
\newtheorem{np*}{Non-Proof}
\newtheorem*{remstar}{Remark}
\theoremstyle{definition}
\newtheorem{defn}[thm]{Definition}
\newtheorem*{defnstar}{Definition}
\newtheorem{exam}[thm]{Example}
\newtheorem*{examstar}{Example}
\newcommand{\pd}[2]{\frac{\partial #1}{\partial #2}}
\newcommand{\pdtwo}[2]{\frac{\partial^2 #1}{\partial #2^2}}
\def\Ind{\setbox0=\hbox{$x$}\kern\wd0\hbox to 0pt{\hss$\mid$\hss} \lower.9\ht0\hbox to 0pt{\hss$\smile$\hss}\kern\wd0}
\def\Notind{\setbox0=\hbox{$x$}\kern\wd0\hbox to 0pt{\mathchardef \nn=12854\hss$\nn$\kern1.4\wd0\hss}\hbox to 0pt{\hss$\mid$\hss}\lower.9\ht0 \hbox to 0pt{\hss$\smile$\hss}\kern\wd0}
\def\ind{\mathop{\mathpalette\Ind{}}}
\def\nind{\mathop{\mathpalette\Notind{}}} 
\newcommand{\m}{\mathbb }
\newcommand{\mc}{\mathcal }
\newcommand{\mf}{\mathfrak }

\def \d {\delta}
\def \NN {\mathbb N}
\def \QQ {\mathbb Q}
\def \ZZ {\mathbb Z}
\def \et {d_{D/\D}}
\def \w {\omega}
\def \D {\Delta}
\def \I {\mathcal{I}}

\title[On linear dependence and completeness]{On linear dependence over complete differential algebraic varieties}
\author[James Freitag]{James Freitag*} \thanks{*This material is based upon work supported by an American Mathematical Society Mathematical Research Communities award and the National Science Foundation Mathematical Sciences Postdoctoral Research Fellowship, award number 1204510.} 
\author[Omar Le\'on S\'anchez]{Omar Le\'on S\'anchez}  
\author[William Simmons]{William Simmons**} \thanks{**This material is based upon work supported by an American Mathematical Society Mathematical Research Communities award.}
\address{freitag@math.berkeley.edu \\
Department of Mathematics\\
University of California, Berkeley\\
970 Evans Hall\\
Berkeley, CA 94720-3840 }
\address{oleonsan@math.mcmaster.ca \\
Department of Mathematics and Statistics \\
McMaster University \\
1280 Main St W \\
Hamilton, ON L8S 4L8}
\address{simmons@math.ucla.edu \\
Mathematics Department\\
University of California, Los Angeles\\
Math Sciences Building 6363\\
Los Angeles, CA 90095}
\date{\today}

\begin{abstract}
We extend Kolchin's results from \cite{KolchinDiffComp} on linear dependence over projective varieties in the constants, to linear dependence over arbitrary complete differential varieties. We show that in this more general setting, the notion of linear dependence still has necessary and sufficient conditions given by the vanishing of a certain system of differential-polynomials equations. We also discuss some conjectural questions around completeness and the catenary problem.
\end{abstract}

\maketitle

\begin{center}
\end{center}

\section{Introduction}

The study of complete differential algebraic varieties began in the 1960's during the development of differential algebraic geometry, and received the attention of various authors over the next several decades \cite{MorrisonSD, BlumComplete, BlumExtensions, KolchinDiffComp}. On the other hand, even though model theory had significant early interactions with differential algebra, it was not until recently that the topic has been the subject of various works using the model-theoretic perspective \cite{PongDiffComplete2000, DeltaCompleteness, SimmonsThesis, PillayDvar}.

Up until the last couple of years, relatively few examples of complete differential algebraic varieties were known. The development of a valuative criterion for completeness led to a variety of new examples, see for instance \cite{PongDiffComplete2000} and \cite{DeltaCompleteness}. Subsequently, more examples have been discovered \cite{SimmonsThesis} using various algebraic techniques in conjunction with the valuative criterion (in Section \ref{compbound} we present an example of the third author's, which shows that there are zero-dimensional projective differential algebraic varieties which are not complete). 

To verify that a given projective differential variety $V$ is complete one has to verify that for any quasiprojective differential variety $Y$, the second projection: 
\begin{equation}\label{diag}
\pi_2 \colon V \times Y \to Y
\end{equation}
is a closed map. After reviewing some basic facts on completeness in Section \ref{compbound}, we establish, by means of the differential-algebraic version of Bertini's theorem, that in order to verify completeness it suffices to check that the above projection maps are \emph{semi-closed}.

Differential completeness is a fundamental notion in differential algebraic geometry, but, except for \cite{KolchinDiffComp}, there has been no discussion of applications of the idea (outside of foundational issues). In Section \ref{lindep}, we consider the notion of linear dependence over an arbitrary projective differential variety. This is a generalization of a notion studied by Kolchin \cite{KolchinDiffComp} in the case of projective algebraic varieties, which in turn generalizes linear dependence in the traditional sense. We prove several results extending the work in \cite{KolchinDiffComp}; for instance, we see that this general notion of linear dependence also has necessary and sufficient conditions given by the vanishing of differential algebraic equations (when working over a complete differential variety).

In the case of the field of meromorphic functions (on some domain of $\m C$) and the projective variety $\m P^n (\m C)$, Kolchin's results \cite{KolchinDiffComp} specialize to the classical result: any finite collection of meromorphic functions is linearly dependent over $\m C$ if and only if the Wronskian determinant of the collection vanishes. There are generalizations of this in several directions; for instance, in the context of multiple variables (i.e., partial differential fields), fully general results on Wronskians and linear dependence of meromorphic functions are relatively recent. Roth \cite{Roth} first established these type of results in the case of rational functions in several variables for use in diophantine approximation. Later his results were generalized to meromorphic functions in some domain of $\m C^m$ via \cite{Wolsson} and \cite{BerensteinChangLi}. It is worth noting that the proofs of these results are analytic in nature. 

In Section \ref{gene}, we point out how our results on linear dependence over arbitrary complete differential varieties generalize the above results in two essential ways: the differential field is not assumed to be a field of meromorphic functions and the linear dependence is considered over an arbitrary solution set to some differential equations (rather than over $\m C^n$). 

In \cite{DeltaCompleteness}, it was established that every complete differential variety is \emph{zero-dimensional} (earlier, this result was established in the ordinary case \cite{PongDiffComplete2000}). Thus, it is natural to ask the following question:
 
\begin{questar} 
To verify completeness, can one restrict to taking products of the given differential variety with zero-dimensional differential varieties? In other words, is it enough to only consider zero-dimensional varieties $Y$ in (\ref{diag})?
\end{questar} 

Although we are not able to give a full answer to this question, we show that it has a positive answer under the additional assumption of the \emph{weak catenary conjecture}. The conjecture is itself a very natural problem in differential algebraic geometry, and the conditional answer to the above question helps motivate the conjecture further. This weak catenary-type conjecture is an easy consequence of the \emph{Kolchin catenary conjecture}, which has been verified in numerous cases, but not in entire generality. 

Section \ref{catsection} is intended partly as a survey on the progress of the catenary conjecture, and partly as an opportunity to pose stronger forms of the conjecture that are interesting in their own right. More precisely, after discussing the catenary problem, we formulate a stronger version for algebraic varieties and show the equivalence of this stronger version to certain maps of prolongation spaces being open. This gives the equivalence of these strong forms of the Kolchin catenary problem to a problem purely in the realm of scheme theory. The proof of the equivalence uses recent work of Trushin \cite{Trushin} on a transfer principal between the Kolchin and Zariski topology called \emph{inheritance}.

\
\noindent {\bf Acknowledgements.} The authors began this work during a visit to the University of Waterloo, which was made possible by a travel grant from the American Mathematical Society through the Mathematical Research Communities program. We gratefully acknowledge this support which made the collaboration possible. We would also like to thank Rahim Moosa for numerous useful conversations during that visit and afterwards.

\section{Projective differential algebraic varieties} 

In this section we review the basic notions and some stantard results on (projective) differential algebraic geometry. For a thorough development of the subject see \cite{KolchinDAAG} or \cite{KolchinDiffComp}. We fix a differentially closed field $(\mathcal{U},\Delta)$ of characteristic zero, where $$\Delta=\{\delta_1,\dots,\delta_m\}$$ is the set of $m$ commuting derivations. We assume $\mathcal{U}$ to be a universal domain for differential algebraic geometry; in model-theoretic terms, we are simply assuming that $(\mc U,\Delta)$ is a sufficiently large saturated model of the theory $DCF_{0,m}$. Throughout $K$ will denote a (small) differential subfield of $\mathcal{U}$. 

A subset of $\m A^n=\m A^n(\mathcal U)$ is \emph{$\Delta$-closed}, or simply closed when the context is clear, if it is the zero set of a collection of $\Delta$-polynomials over $\mc U$ in $n$ differential indeterminates (these sets are also called \emph{affine} differential algebraic varieties).  When the collection of $\Delta$-polynomials defining a $\Delta$-closed set is over $K$, we say that the $\Delta$-closed is defined over $K$. 

Following the standard convention, we will use $K \{ y_0, y_1 , \ldots , y_n \}$ to denote the ring of \emph{$\Delta$-polynomials} over $K$ in the $\Delta$-indeterminates $y_0, y_1, \ldots , y_n$.

\begin{defn}  A (non-constant) $\Delta$-polynomial $f$ in $K\{ y_0 , \ldots , y_n \}$ is \emph{$\Delta$-homogeneous of degree d} if 
$$f(t y_0, \ldots t y_n \} = t^d f(y_0 , \ldots , y_n ),$$ where $t$ is another $\Delta$-indeterminate. 
\end{defn}

The reader should note that $\Delta$-homogeneity is a stronger notion that homogeneity of a differential polynomial as a polynomial in the algebraic indeterminates $\delta_m^{r_m}\cdots\delta_1^{r_1}y_i$. For instance, for any $\delta\in \Delta$, $$\delta y - y$$ is a homogeneous $\Delta$-polynomial, but not a $\Delta$-homogeneous $\Delta$-polynomial. The reader may verify that the following is $\Delta$-homogeneous:
$$y_1 \delta y_0  -y_0 \delta y_1 -y_0y_1.$$

Generally, we can easily homogenize an arbitrary $\Delta$-polynomial in $y_1,\dots,y_n$ with respect to a new $\Delta$-variable $y_0$. Let $f$ be a $\Delta$-polynomial in $K\{y_1,\dots,y_n\}$, then for $d$ sufficiently large $y_0^d f(\frac{y_1}{y_0},\dots,\frac{y_n}{y_0})$ is $\Delta$-homogeneous of degree $d$. For more details and examples see \cite{PongDiffComplete2000}.

As a consequence of the definition, the vanishing of $\Delta$-homogeneous $\Delta$-polynomials in $n+1$ variables is well-defined on $\m P^n=\m P^n(\mathcal U)$. In general, the $\Delta$-closed subsets of $\m P^n$ defined over $K$ are the zero sets of collections of $\Delta$-homogeneous $\Delta$-polynomials in $K \{y_0 , \ldots , y_n \}$ (also called \emph{projective} differential algebraic varieties). Furthermore, $\Delta$-closed subsets of $\m P^n  \times \m A^m$, defined over $K$, are given by the zero sets of collections of $\Delta$-polynomials in $$K \{ y_0 , \ldots , y_n, z_1, \ldots , z_m \}$$ which are $\Delta$-homogeneous in $(y_0,\dots,y_n)$.

\subsection{Dimension polynomials for projective differential algebraic varieties} \label{dimpol}

Take $\alpha \in \m P ^n$ and let $\bar a = ( a_0 , \ldots , a_n ) \in \m A ^{n+1}$ be a representative for $\alpha$. Choose some index $i$ for which $a_i \neq 0$. The field extensions $K\left(\frac{a_0}{a_i},\ldots,\frac{a_n}{a_i}\right)$ and $K \left\langle \frac{a_0}{a_i},\ldots,\frac{a_n}{a_i} \right\rangle$ do not depend on which representative $\bar a$ or index $i$ we choose. Here $K\langle \bar a\rangle$ denotes the $\Delta$-field generated by $\bar a$ over $K$.

\begin{defn} With the notation of the above paragraph, the {\it Kolchin polynomial of $\alpha$ over $K$} is defined as
$$\omega _{ \alpha /K } (t) = \omega _{\left(\frac{a_0}{a_i}, \dots,\frac{a_n}{a_i}\right) /K} (t) ,$$
where $\omega _{\left(\frac{a_0}{a_i}, \dots,\frac{a_n}{a_i}\right) /K} (t)$ is the standard Kolchin polynomial of $\left(\frac{a_0}{a_i}, \dots,\frac{a_n}{a_i}\right)$ over $K$ (see Chapter II of \cite{KolchinDAAG}). The \emph{$\Delta$-type of $\alpha$} over $K$ is defined to be the degree of $\omega_{\alpha/K}$. By the above remarks, these two notions are well-defined; i.e., they are independent of the choice of representative $\bar a$ and index $i$.
\end{defn} 

Let $\beta \in \m P^n$ be such that the closure (in the $\Delta$-topology) of $\beta$ over $K$ is contained in the closure of $\alpha$ over $K$. In this case we say that $\beta$ is a \emph{differential specialization} of $\alpha$ over $K$ and denote it by $\alpha \mapsto_K \beta$. Let $\bar b$ be a representative for $\beta$ with $b_i \neq 0$. Then, by our choice of $\beta$ and $\alpha$, if $\bar a $ is a representative of $\alpha $, then $a_i \neq 0$; moreover, the tuple $\left(\frac{b_0}{b_i},\dots,\frac{b_n}{b_i}\right)$ in $\mathbb{A}^{n+1}$ is a differential specialization of $\left(\frac{a_0}{a_i}, \dots,\frac{a_n}{a_i}\right)$ over $K$. When $V \subseteq \m P^n$ is an irreducible $\Delta$-closed set over $K$, then a \emph{generic point} of $V$ over $K$ (when $K$ is understood we will simply say generic) is simply a point $\alpha\in V$ for which $V=\{\beta \, | \, \alpha \mapsto_K \beta \}$. It follows, from the affine case, that every irreducible $\Delta$-closed set in $\m P^n$ has a generic point 
over $K$, and that any two such generics have the same isomorphism type 
over $K$. 

\begin{defn} Let $V\subseteq \m P^n$ be an irreducible $\Delta$-closed set. 
The {\it Kolchin polynomial} of $V$ is defined to be $$\omega _{ V} (t) = \omega_ {\alpha /F} (t),$$ where $F$ is any differential field over which $V$ is defined and $\alpha$ is a generic point of $V$ over $F$. It follows, from the affine case, that $\omega_V$ does not depend on the choice of $F$ or $\alpha$. The \emph{$\Delta$-type of $V$} is defined to be the degree of $\omega_V$.
\end{defn} 

\begin{rem} Let $V\subseteq \m P^n$ be an irreducible $\Delta$-closed set, and recall that $m$ is the number of derivations.
\begin{itemize}
\item [(i)] $V$ has $\Delta$-type $m$ if and only if the differential function field of $V$ has positive differential transcendence degree.
\item [(ii)] $V$ has $\Delta$-type zero if and only if the differential function field of $V$ has finite transcendence degree.
\end{itemize}
\end{rem}

The {\it dimension} of $V$, denoted by $\operatorname{dim}V$, is the differential transcendence degree of the differential function field of $V$. Thus, by a zero-dimensional differential variety we mean one of $\Delta$-type less than $m$ (in model-theoretic terms this is equivalent to the Lascar rank being less than $\omega^{m}$).
 
In various circumstances it is advantageous (and will be useful for us in Section \ref{lindep}) to consider $\m P^n$ as a quotient of $\m A^{n+1}$. For example, if $\mf p$ is the $\Delta$-ideal of $\Delta$-homogeneous $\Delta$-polynomials defining $V \subseteq \m P^{n}$ and we let $W \subseteq \m A^{n+1}$ be the zero set of $\mf p$, then 
\begin{equation}\label{use}
\omega _{W} (t)=\omega _ {V} (t) + \binom{m+t}{m},
\end{equation}
where the polynomial on the left is the standard Kolchin polynomial of $W$ (see \S5 of \cite{KolchinDiffComp}).

\section{On completeness}\label{compbound}

In this section we  recall a few facts and prove some foundational results on complete differential algebraic varieties (for more basic properties we refer the reader to \cite{DeltaCompleteness}). We start by recalling the definition of $\Delta$-completeness.

\begin{defn}\label{maindef}
A $\Delta$-closed $V \subseteq \m P^n$ is \emph{$\Delta$-complete} if the second projection $$\pi_2: V \times Y \rightarrow Y$$ is a $\Delta$-closed map for every quasiprojective differential variety $Y$. Recall that a quasiprojective differential variety is simply an open subset of a projective differential variety.
\end{defn}

We will simply say \emph{complete} rather than $\Delta$-complete. This should cause no confusion with the analogous term from the algebraic category because we will work exclusively in the category of differential algebraic varieties. 

The first differential varieties for which completeness was established were the constant points of projective algebraic varieties \cite{KolchinDiffComp}. One might attempt to establish a variety of examples via considering algebraic D-variety structures on projective algebraic varieties; in Lemma \ref{Dvarst} below, we prove that indeed the sharp points of an algebraic D-variety is a complete differential variety. 

Let us first recall that an algebraic D-variety is a pair $(V,\mathcal D)$ where $V$ is an algebraic variety and $\mathcal D$ is a set of $m$ commuting derivations on the structure sheaf $\mathcal O_V$ of $V$ extending $\Delta$. A point $v\in V$ is said to be a \emph{sharp point} of $(V, \mathcal D)$ if for every affine neighborhood $U$ of $v$ and $f\in \mathcal O_V(U)$ we have that $D(f)(v)=\delta(f(v))$ for all $D\in \mathcal D$ and $\delta \in \Delta$. The set of all sharp points of $V$ is denoted by $(V,\mathcal D)^\sharp$. It is worth noting that given an algebraic variety $W$ defined over the constants, one can equip $W$ with a canonical D-variety structure $\mathcal D_0$ (by setting each $D_i$ to be the unique extension of $\delta_i$ that vanishes on the affine coordinate functions) such that $(W,\mathcal D_0)^\sharp$ is precisely the set of constant points of $W$. We refer the reader to \cite{Buium1} for basic properties of algebraic D-varieties.

\begin{lem}\label{Dvarst}
If $(V,\mathcal D)$ is an algebraic D-variety whose underlying variety V is projective, then $(V,\mathcal D)^\#$ is a complete differential variety.
\end{lem} 
\begin{proof}
In \cite{Buium}, Buium proves that if $V$ is projective then $(V,\mathcal D)$ is isotrivial. That is, there is an isomorphism $f:V\to W$ of algebraic varieties with $W$ defined over the constants such that the image of $(V,\mathcal D)^\sharp$ under $f$ is precisely the set of constant points $W$. Thus, $(V,\mathcal D)^\sharp$ is isomorphic to a projective algebraic variety in the constants. The latter we know is complete, and hence $(V,\mathcal D)^\#$ is complete.
\end{proof}

We now recall a class of examples developed in \cite{PongDiffComplete2000}, which show the existence of complete differential varieties that are not isomorphic to algebraic varieties in the constants.

\begin{exam}\label{exam5}
Restrict to the case of a single derivation $(\mc U,\delta)$. Let $V$ be the $\delta$-closure of $\delta y= f(y)$ in $\m P^1,$ where $f(y) \in K[y]$ is of degree greater than one. In \cite{PongDiffComplete2000}, it was shown that $V$ is a complete differential variety. Under the additional assumption that $f(y)$ is over the constants, by a theorem of McGrail \cite[Theorem 2.8]{McGrail} and Rosenlitch \cite{notmin}, it is well understood when such a differential variety is non isomorphic to an algebraic variety in the constants: 

\begin{fact}\cite[page 71]{MMP}
Suppose that $f(y)$ is a rational function over the constants of a differential field $(K,\delta)$. Then $V = \{ x \in \m A^1 \, | \, \delta x = f(x) \}$ is isomorphic to an algebraic variety in the constants if and only if either: 
\begin{enumerate} 
\item  $\frac{1}{f(y)} = c \frac{\pd{u}{y}}{u}$ for some rational function $u$ over the constants and $c$ a constant.  
\item $\frac{1}{f(y)} = c \pd{v}{y}$ for some rational function $v$ over the constants and $c$ a constant. 
\end{enumerate} 
\end{fact} 
The previous two results yield a large class of complete differential varieties which are non isomorphic to an algebraic variety in the constants. Beyond order one, nonlinear equations are rather difficult to analyze with respect to completeness, because the valuative criteria developed in \cite{PongDiffComplete2000, DeltaCompleteness} are difficult to apply. For such examples, we refer the reader to \cite{SimmonsThesis}. 
\end{exam} 

Note that in all the above examples of complete differential varieties, the varieties are zero-dimensional. This is generally true; in \cite{DeltaCompleteness} the following was established: 

\begin{fact}\label{zerodim} 
Every complete differential variety is zero-dimensional.
\end{fact} 

This implies that the completeness question in differential algebraic geometry is one which only makes sense to ask for zero-dimensional projective differential varieties. Thus, it seems natural to inquire whether the notion can be completely restricted to the realm of zero-dimensional differential varieties. More precisely, a priori, the definition of completeness requires quantification over all differential subvarieties of the product of $V$ with an arbitrary quasiprojective differential variety $Y$. In light of Fact~\ref{zerodim}, it seems logical to ask if one can restrict to zero-dimensional $Y$'s for the purposes of verifying completeness. We provide some insight on this question in Section \ref{catsection}; of course, the positive answer to such question would be helpful for answering the following:
\begin{ques} Which projective differential varieties are complete?
\end{ques}

In general the above question is rather difficult. As the following example shows, even at the level of $\Delta$-type zero one can find incomplete differential varieties.

\begin{exam}
Recently, the third author \cite{SimmonsThesis} constructed the first known example of a zero-dimensional projective differential algebraic variety which is not complete (in fact of differential type zero). Restrict to the case of a single derivation $(\mathcal{U},\delta)$. Consider the subset $W$ of $\mathbb{A}^{1}\times \mathbb{A}^{1}$ defined by $x''=x^{3}$ and $2yx^{4}-4y(x')^{2}=1$. One may check (by differentially homogenizing $x''=x^{3}$ and observing that the point at infinity of $\mathbb{P}^{1}$ does not lie on the resulting variety) that $x''=x^{3}$ is already a projective differential variety. Thus, $W$ is a $\delta$-closed subset of $\mathbb{P}^{1}\times \mathbb{A}^{1}$. A short argument (see \cite{SimmonsThesis} for details) establishes that $\pi_{2}(W)$ is the set $\{y\mid y'=0 \text{ and } y\neq 0\}$ which is not $\delta$-closed. 

Let us point out that this example works because the derivative of $2x^{4}-4(x')^{2}$ belongs to the $\delta$-ideal generated by $x''-x^{3}$ though $2x^{4}-4(x')^{2}$ itself does not. Since determining membership in a $\delta$-ideal is often difficult, it is quite possible that identifying complete differential varieties is inherently a hard problem. This highlights the need for reductions in the process of checking completeness.  
\end{exam}

We finish this section by proving that in the definition of completeness one can slightly weaken the requirement of the second projection maps being closed maps. In order to state our result, we need the following definition.

\begin{defn}
A morphism $f\colon X\to Y$ of differential algebraic varieties is said to be \emph{semi-closed}, if for every $\Delta$-closed subset $Z\subseteq X$ either $f(Z)$ is $\Delta$-closed or $f(Z)^{cl}\setminus f(Z)$ is positive-dimensional. Here $f(Z)^{cl}$ denotes the $\Delta$-closure of $f(Z)$.
\end{defn}

Note that if at least one of $X$ or $Y$ is zero-dimensional, then a morphism $f:X\to Y$ is semi-closed iff it is closed. 

Recall that a ($\Delta$-)\emph{generic hyperplane} of $\mathbb A^n$ is the zero set of a polynomial of the form 
$$f(y_1,\ldots,y_n)=a_0+a_1y_1+\cdots+a_ny_n,$$
where the $a_i$'s are $\Delta$-algebraically independent. We will make use of the following differential version of Bertini's theorem, which appears in \cite{JBertini}. 

\begin{fact} \label{Bertini} 
Let $V\subseteq \mathbb{A}^n$ be an irreducible differential variety of dimension $d$ with $d>1$, and let $H$ be a generic hyperplane of $\mathbb{A}^n$. Then $V\cap H$ is irreducible of dimension $d-1$, and its Kolchin polynomial is given by
$$\omega_{(V\cap H)}(t)=\omega_V (t)- \binom{m+t}{m} .$$
\end{fact} 

We can now prove

\begin{prop}\label{mainthm} 
A $\Delta$-closed $V \subseteq \m P^n$ is complete if and only if $\pi_2: V \times Y \rightarrow Y$ is a semi-closed map for every quasiprojective differential variety $Y.$  
\end{prop} 
\begin{proof} 
Towards a contradiction, suppose that $\pi_2:V\times Y\to Y$ is semi-closed for every quasiprojective $Y$, but $V$ is not complete. Then there must be some $Y_1$, a positive dimensional quasiprojective differential variety, such that $\pi_2: V \times Y_1 \rightarrow Y_1$ is not closed. Because the question is local, we can assume that $Y_1$ is affine and irreducible. Let $X_1 \subseteq V \times Y_1$ be a $\Delta$-closed set such that $\pi_2 (X)$ is not closed. We may assume that $X_1$ is irreducible and, because of our semi-closedness assumption, that $\pi_2(X_1)^{cl}\setminus \pi_2(X_1)$ is positive dimensional. Let $W_1$ be the $\Delta$-closure of $\pi_2(X_1)^{cl}\setminus \pi_2(X_1)$. Note that $W_1$ is positive dimensional and, since the theory $DCF_{0,m}$ admits quantifier elimination, it has strictly smaller Kolchin polynomial than $\pi(X_1)^{cl}$.

Now, let $H$ be a generic hyperplane (generic over a differentially closed field over which everything else is defined) in $\m A^n$, where $n$ is such that $Y_1\subseteq \m A^n$. By Fact \ref{Bertini}, $W_1\cap H\neq \emptyset$. Now let $Y_2=Y_1\cap H$ and consider 
$$X_2:= X_1 \cap (V \times Y_2) \subseteq V \times Y_2.$$ 
We claim that $\pi_2 (X_2)=\pi_2(X_1)\cap H$ is not closed. Suppose it is, then, as $W_2:= W_1\cap H\neq \emptyset$, $\pi_2(X_1)\cap H$ is a closed proper subset of $\pi_2(X_1)^{cl}\cap H$. 
The fact that $\emptyset \neq W_1\cap H\subset \pi_{2}(X_1)^{cl}\cap H$ (which follows from Fact \ref{Bertini} and noting that Kolchin polynomials behave predictably with respect to intersections with generic hyperplanes) contradicts irreducibility of $\pi_2(X_1)^{cl}\cap H$. Thus, $\pi(X_2)=\pi_2(X)\cap H$ is not closed. Further, by Fact \ref{Bertini} again, the dimension of $Y_2$, $X_2$ and $W_2$ is one less that the dimension of $Y_1$, $X_1$ and $W_1$, respectively. 

Since this process decreases the dimensions, after a finite number of steps it would yield $Y$, $X$ and $W\neq \emptyset$, with $X$ a closed subset of $V\times Y$ such that $W=\pi_2(X)^{cl}\setminus\pi_2(X)$ is zero-dimensional. This contradicts semi-closedness, and the result follows.
\end{proof}

\section{Linear dependence over differential algebraic varieties}\label{lindep}

In this section we extend Kolchin's results from \cite{KolchinDiffComp} on linear dependence over projective varieties in the constants, to linear dependence over arbitrary complete differential varieties. We begin by giving a natural definition of linear dependence over an arbitrary projective differential variety.

\begin{defn} Let $V \subseteq \m P^n$ be $\Delta$-closed. We say that $\bar a=(a_0,\ldots,a_{n})\in \m A^{n+1}$ is {\it linearly dependent over $V$} if there is a point $v = [v_0: \cdots : v_n]  \in V$ such that $\sum_{j=0}^n v_i a_i =0$. Similarly, we say that $\alpha\in\m P^n$ is linearly dependent over $V$ if there is a representative $\bar a$ of $\alpha$ such that $\bar a$ is linearly dependent over $V$. Note that it does not matter which particular representative of $\alpha$ or $v$ we choose when testing to see if $v$ witnesses the $V$-linear dependence of $\alpha$. 
\end{defn} 

In \cite{KolchinDiffComp}, Kolchin states the following problem:
\begin{exam} \label{Mainexam} Consider an irreducible algebraic variety $V$ in $\m P^n(\m C)$ for some $n$. Let $f_0, f_1, \ldots , f_n$ be meromorphic functions in some region of $\m C$. Ritt once remarked (but seems to have not written down a proof) that there is an ordinary differential polynomial $R \in \m C \{y_0 , y_1, \ldots , y_n \}$ which depends only on $V$ and has order equal to the dimension of $V$ such that a necessary and sufficient condition that there is $c \in V$ with $ \sum c_i f_i =0$ is that $(f_0 , f_1 , \ldots ,f_n)$ be in the general solution of the differential equation $R(y_0 , y_1 , \ldots , y_n )=0$. For a more thorough discussion, see \cite{KolchinDiffComp}. 
\end{exam} 

It is natural to ask if in the previous example one can replace the algebraic variety $V$ with an arbitrary complete differential variety. Kolchin offers a solution to this question in the case when $V$ is an arbitrary projective differential variety living inside the constants. However, the fact that such projective varieties (viewed as zero-dimensional differential algebraic varieties) are complete in the $\Delta$-topology turns out to be the key to proving the existence of the above differential polynomial $R$. 

We will extend Kolchin's line of reasoning for proving the assertion in Example \ref{Mainexam}. Namely, we start with a complete differential algebraic variety, rather than the constant points of an algebraic variety. Recall that, as we pointed out in Example \ref{exam5}, there are many complete differential varieties that are not isomorphic to the constant points of an algebraic variety.

\begin{thm}\label{linear}
Let $V \subset \m P^n$ be a complete differential variety defined and irreducible over $K$. Let $$ld(V):= \{x \in \m P^n \, | \,  x \text{ is linearly dependent over } V \}.$$ Then $ld(V)$ is an irreducible differential subvariety of $\m P^n$ defined over $K$. 
\end{thm} 

\begin{rem} With the correct hypotheses, Kolchin's proof of the special case essentially goes through here.  Similar remarks apply to the strategy of the next proposition and corollary, where Kolchin's original argument provides inspiration. For the proposition, the result seems to require a few new ingredients, mainly doing calculations in the generic fiber of the differential tangent bundle. 
\end{rem} 

\begin{proof} 
Let $\mf p \subseteq  K\{\bar z\}=K \{ z_0, z_1, \ldots ,z_n \}$ be the differential ideal corresponding to $V$. 
Now, let $\bar y = (y_0, y_1, \ldots y_n )$ and consider the differential ideal $\mf p_1 \in K\{\bar z , \bar y \}$ given by 
$$\left[ \mf p, \sum_{j=0}^n y_j z_j  \right]: (\bar y \bar z ) ^ \infty$$ 
which by definition is 
$$\left\{f\in K\{\bar z,\bar y\}: \, (y_iz_j)^ef\in \left[\mf p,\sum y_jz_j\right],\,  i,j \leq n,\text{ for some } e\in \m N \right\}.$$
As $\mf p_1$ is differentially bi-homogeneous, it determines a (multi-)projective differential algebraic variety $W \subseteq \m P^n \times \m P^n$. It is clear that the coordinate projection maps have the form
$$
\xymatrix{
&  & W \ar[ld]_{\pi_1} \ar[rd]^{\pi_2} &  & \\
 &   V &  & ld(V) &  
}
$$

Further, $ld(V)= \pi_2 (W)$, and, since $V$ is complete, $ld(V)$ is closed in the Kolchin topology of $\m P^n$ and defined over $K$. 

Next we prove that $ld(V)$ is irreducible over $K$. Let $\bar a=(a_0,\dots,a_n)$ be a representative of a generic point of $V$ over $K$ and fix $j$ such that $a_j \neq 0$. Pick elements $u_k \in \mc U$ for $0 \leq k \leq n$ and $k \neq j$ which are $\Delta$-algebraically independent over $K \langle \bar a \rangle $. Let 
$$u_j = - \sum _{k \neq j} u_k a_j ^{-1} a_k,$$ 
and $\bar u = ( u_0 , \ldots , u_n)$. One can see that $ (\bar a , \bar u)$ is a representative of a point in $W \subseteq \m P^n \times \m P^n$, so that $[u_0 : \cdots : u_n ] \in ld(V)$. 

We claim that $[u_0:\cdots:u_n]$ is a generic point of $ld(V)$ over $K$ (this will show that $ld(V)$ is irreducible over $K$). To show this it suffices to show that $(\bar a,\bar u)$ is a generic point of $\mf p_1$; i.e., it suffices to show that for every $p\in K\{\bar z,\bar y\}$ if $p(\bar a,\bar u)=0$ then $p\in \mf p_1$.

Let $ p \in K \{\bar z , \bar y \}$ be any differential polynomial. By the differential division algorithm there exists $p_0 \in K \{ \bar z, \bar y \}$ not involving $y_j$ such that 
$$z_j^e  p \equiv p_0 \quad  mod \left[\sum_{0 \leq i \leq n } y_i z_i \right]$$ 
for some $e \in \m N$. 
Thus, we can write $p_0$ as a finite sum $\sum p_M M$ where each $M$ is a differential monomial in $(y_k ) _{ 0 \leq k \leq n , k \neq j}$ and $p_M \in K \{ \bar z \}.$ 
Now, as $(u_k ) _{ 0 \leq k \leq n , k \neq j}$ are differential algebraic independent over $K \langle \bar a \rangle$, it follows that if $p ( \bar a , \bar u)=0$ then $p_M (\bar a) =0$ for all $M$ (and hence $p_M\in \mf p$, since $\bar a$ is a generic point of $\mf p$). This implies that, if $p(\bar a,\bar u)=0$, $p_0\in {\mf p} \cdot K\{\bar z,\bar y\}$ and so $p\in \mf p_1$, as desired.

\end{proof} 

For the following proposition we will make use of the following fact of Kolchin's about the Kolchin polynomial of the differential tangent space (we refer the reader to \cite{KolchinDAG} for the definition and basic properties of differential tangent spaces).

\begin{fact}\label{tang}
Let $V$ be an irreducible differential algebraic variety defined over $K$ with generic point $\bar v$. Then 
$$\omega_{T^{\Delta}_{\bar v}V}=\omega_V,$$
where $T^{\Delta}_{\bar v}V$ denotes the differential tangent space of $V$ at $\bar v$.
\end{fact}

In the case when the complete differential algebraic variety $V$ has $\Delta$-type zero, we have the following result on the Kolchin polynomial of $ld(V)$.

\begin{prop}\label{prok}
Let $V\subset \m P^n$ be a complete differential variety defined and irreducible over $K$. If $V$ has constant Kolchin polynomial equal to $d$, then the Kolchin polynomial of $ld(V)$ is given by
$$\omega _{ld(V)} (t) = (n-1) \binom{m+t}{m} + d.$$
\end{prop}

\begin{rem} In the ordinary case, the hypotheses on the Kolchin polynomial do not constitute any assumption on $V$. In that case, every such complete $V$ has differential type zero. In the partial case, the situation is much less clear. It is not known whether every complete differential algebraic variety has constant Kolchin polynomial (see \cite{DeltaCompleteness} for more details).
\end{rem} 

\begin{proof}
Let $W \subseteq \m P^{n} \times \m P^{n}$, $\bar a$ and $\bar u$ be as in the proof of Theorem~\ref{linear}. Fix $j$ such that $a_j\neq 0$ and, moreover, assume that $a_j =1$. Now, write $\bar a^*$ and $\bar u^*$ for the tuples obtained from $\bar a$ and $\bar u$, respectively, where we omit the $j$-th coordinate. Let $W_1 \subseteq \m A^{n} \times \m A^{n+1}$ be the differential algebraic variety with generic point $(\bar a^* , \bar u )$ over $K$. Consider $T_{( \bar a^* , \bar u )} ^\Delta W_1$, the differential tangent space of $W_1$ at $(\bar a^*,\bar u)$. Let $(\bar \alpha , \bar \eta )$ be generic of $T^{\Delta}_{(\bar a^*,\bar u)}W_1$  over $K\langle \bar a^*, \bar u \rangle$. 

From the equation $y_j = - \sum _{k \neq j} y_k z_k$ satisfied by $(\bar a^*, \bar u)$, we see that
$$-\sum_{ i \neq j } a_i \eta _i  - \eta_j  = \sum _ {i \neq j} u_i \alpha _i .$$

Choose $d_1$ so that $|\Theta (d_1)|$ is larger than $n \cdot |\Theta (d) |$, where $\Theta(d_1)$ denotes the set of derivative operators of order at most $d_1$ (similarly for $\Theta(d)$). 

For $\theta \in \Theta (d_1)$, we have that
$$\theta \left(-\sum_{ i \neq j } a_i \eta _i  - \eta_j  \right) = \theta \left( \sum _ {i \neq j} u_i \alpha _i \right) .$$
In the expression on the right $\theta \left( \sum _ {i \neq j} u_i \alpha _i \right)$, consider the coefficients of $\theta \alpha_i$ in terms of $\theta \bar u^*$, and denote these coefficients by $f (  i ,\theta' , \theta  ) \in K \langle \bar u ^ * \rangle$; i.e., $f (  i ,\theta' , \theta  )$ is the coefficient of $\theta'  \alpha_{i}$ in the equation 
$$\theta \left( -\sum_{ i \neq j } a_i \eta _i  - \eta_j  \right) = \theta \left( \sum _ {i \neq j} u_i \alpha _i \right) .$$
We can thus express this equation in the form
$$\left( \begin{array}{c} \theta \left(-\sum_{ i \neq j } a_i \eta _i  - \eta_j   \right)  \end{array} \right) =  F_\theta  A$$ 
where $F_\theta$ is the row vector with entries $$(f(i , \theta', \theta))_{ i \neq j,  \theta' \in \Theta(d)}$$ and $A$ is the column vector with entries 
$$\left( \begin{array}{c} \theta' \alpha_i \end{array} \right)_{i \neq j, \theta' \in \Theta (d) }. $$
Note that we only have the vectors $F_\theta$ and $A$ run through $\theta' \in \Theta (d)$ (instead of all of $\Theta(d_1)$). This is because any derivative of $\alpha_i$, $i\neq j$, of order higher than $d$ can be expressed as a linear combination of the elements of the vector $A$. This latter observation follows from the fact that $\omega_{\alpha/K\langle \bar a^*, \bar u\rangle}=d$, which in turn follows from the facts that $\alpha$ is a generic point of $T^\Delta_{\bar a^*}V$, our assumptioin on $\omega_V$, and Fact~\ref{tang}. 

By the choice of $d_1$ and $\bar u^*$ (recall that $\bar u^*$ consists of independent differential transcendentals over $K\langle \bar a \rangle$), there are $n \cdot |\Theta (d) |$ linearly independent row vectors $F_ \theta$. So, we can see (by inverting the nonsingular matrix which consists of $n \cdot |\Theta (d) |$ such $F_ \theta$'s as the rows) that all the elements of the vector $A$ belong to $K\langle \bar a ^ *, \bar u \rangle (( \theta \bar \eta )_{\theta\in \Theta(d_1)})$. Thus 
$$K\langle \bar a^*, \bar u\rangle((\theta \bar \eta)_{\theta\in \Theta(d_1)})=K\langle \bar a^*,\bar u, \bar\alpha\rangle ((\theta\eta_i)_{i\neq j, \theta\in \Theta(d_1)}).$$
Noting that $\bar \eta^*=(\eta_i)_{i\neq j}$ is a tuple of independent transcendentals over $K\langle \bar a^*, \bar u\rangle$ and that $\bar \eta^* $ is ($\Delta$-)independent from $\bar \alpha$ over $K \langle \bar a ^* , \bar u\rangle$, the above equality means that for all large enough values of $t$, 
$$\omega _ { \bar \eta / K \langle \bar a^* , \bar u \rangle } (t) = n \binom{m+t}{m} + d.$$
  Finally, since $\bar \eta$ is a generic of the differential tangent space at $\bar u$ of the $K$-locus of $\bar u$, Fact~\ref{tang} implies that 
  \begin{equation}\label{kopo}
  \omega_{\bar u/K}=n \binom{m+t}{m} + d,
  \end{equation}
  and thus by equation (\ref{use}) (in Section~\ref{dimpol}), we get
  $$\omega_{ld(V)}=(n-1) \binom{m+t}{m} + d,$$
  as desired.
\end{proof}

\begin{corr}
 Let $V\subset \m P^n$ be a complete differential variety defined and irreducible over $K$, and suppose we are in the ordinary case (i.e., $|\Delta|=1$). If the Kolchin polynomial of $V$ equals $d$, then there exists a unique (up to a nonzero factor in $K$) irreducible $R\in K\{\bar y\}$ of order $d$ such that an element in $\m A^{n+1}$ is linearly independent over $V$ if and only if it is in the general solution of the differential equation $R(\bar y)=0$.
\end{corr}
\begin{proof}
Let $\mf p$ be the differential (homogeneous) ideal of $ld(V)$ over $K$. Then, by equation (\ref{kopo}) in the proof of Proposition~\ref{prok}, we get
$$\omega_{\mf p}=n(t+1)+d=(n+1)\binom{t+1}{1}-\binom{t-d+1}{1}.$$
Thus, by \cite[Chapter IV, \S 7, Proposition 4]{KolchinDAAG}, there exists an irreducible $R\in K\{\bar y\}$ of order $d$ such that $\mf p$ is precisely the general component of $R$; in other words, an element in $\m A^{n+1}$ is linearly independent over $V$ if and only if it is in the general solution of the differential equation $R=0$. For uniqueness, let $R'$ be another differential polynomial over $K$ of order $d$ having the same general component as $R$. Then, by \cite[Chapter IV, \S 6, Theorem 3(b)]{KolchinDAAG}, $R'$ is in the general component of $R$ and so $ord(R')\geq ord(R)$. By symmetry, $R$ is in the general component of $R'$, and so we get that $ord(R)=ord(R')$. Thus $R$ and $R'$ divide each other, as desired.
\end{proof}

We finish this section by discussing how the assertion of Example~\ref{Mainexam} follows from the results of this section. Let $k$ be a differential subfield of $K$.

\begin{defn}
Let $V$ be a differential algebraic variety defined over $k$. We say $V$ is $k$-large with respect to $K$ if $V(k)=V(\bar{K})$ for some (equivalently for every) $\bar{K}$ differential closure of $K$.
\end{defn}

One can characterize the notion of largeness in terms of differential subvarieties of $V$ as follows:

\begin{lem}
$V$ is $k$-large with respect to $K$ if and only if for each differential subvariety $W$ of $V$ defined over $K$, $W(k)$ is $\Delta$-dense in $W$.
\end{lem}
\begin{proof}
Suppose $W(k)$ is $\Delta$-dense in $W$, for each differential algebraic subvariety $W$ of $V$ defined over $K$. Let $\bar{a}$ be a $\bar{K}$-point of $V$. Since $tp(\bar a/K)$ is isolated, there is a differential polynomial $f\in K\{\bar{x}\}$ such that every differential specialization $\bar b$ of $\bar a$ over $K$ satisfying $f(\bar b)\neq 0$ is a generic differential specialization. Let $W\subseteq V$ be the differential locus of $\bar a$ over $K$. By our assumption, there is a $k$-point $\bar b$ of $W$ such that $f(\bar b)\neq 0$. Hence, $\bar b$ is a generic differential specialization of $\bar a$ over $K$, and so $\bar a$ is a $k$-point.

The converse is clear since for every differential algebraic variety $W$, defined over $\bar K$, $W(\bar K)$ is $\Delta$-dense in $W$.
\end{proof}

\begin{rem}\label{remla}
Let $V$ be an (infinite) algebraic variety in the constants defined over $k$, and $K$ a differential field extension of $k$ such that $K^\Delta=k^\Delta$. Here $K^{\Delta}$ and $k^{\Delta}$ denote the constants of $K$ and $k$, respectively.
\begin{enumerate}
\item $V$ is $k$-large with respect to $K$ if and only if $k^{\Delta}$ is algebraically closed. Indeed, if $V$ is $k$-large, the image  of $V(k)$ under any of the Zariski-dominant coordinate projections of $V$ is dense in $\bar K^\Delta$. Hence, $k^\Delta=\bar K^\Delta$, implying $k^\Delta$ is algebraically closed. Conversely, if $k^\Delta$ is algebraically closed, then $k^\Delta=\bar K^\Delta$ (since $k^\Delta=K^\Delta$). Hence, $V(k^\Delta)=V(\bar K^\Delta)$, but since $V$ is in the constants we get $V(k)=V(\bar K)$.
\item In the case $k=\m C$ and $K$ is the field of meromorphic in some region of $\m C$. We have that $K^\Delta=\m C$ is algebraically closed, and so $V$ is $k$-large with respect to $K$. Hence, in Example \ref{Mainexam} the largeness condition of $V$ is given implicitly.
\end{enumerate}
\end{rem}

\begin{lem}\label{large}
Let $V\subset \m P^n$ be a complete differential algebraic variety defined over $k$, and suppose $V$ is $k$-large with respect to $K$. Then the $K$-points of $\m P^n$ that are linearly dependent over $V$ are precisely those that are linearly dependent over $V(k)$.
\end{lem}

\begin{proof}
Suppose $\alpha\in \m P^n(K)$ is linearly independent over $V$. Then, since the models of $DCF_{0,m}$ are existentially closed, we can find $v\in V(\bar K)$, where $\bar K$ is a differential closure of $K$, such that $\sum v_ia_i=0$ where $\bar a$ is a representative of $\alpha$. But, by our largeness assumption, $V(\bar K)=V(k)$, and thus $\alpha$ is linearly independent over $V(k)$. 
\end{proof}

Putting together Theorem \ref{linear}, Remark~\ref{remla}, and Lemma \ref{large}, we see that if $V$ is an irreducible algebraic variety in $\m P^n(\m C)$ and $K$ is a differential field extension of $\m C$ with no new constants then there is a projective differential algebraic variety defined over $\m C$, namely $ld(V)$, which only depends on $V$ such that for any tuple $f=(f_0, f_1,\ldots,f_n)$ from $K$, $f$ is linearly dependent over $V$ if and only if $f\in ld(V)$.

\subsection{Generalized Wronskians}\label{gene}
It is well-known that a finite collection of meromorphic functions (on some domain of $\m C$) is linearly dependent over $\m C$ if and only if its Wronskian vanishes. Roth \cite{Roth} generalized (and specialized) this fact to rational functions in several variables using a generalized notion of the Wronskian. This was later generalized to the analytic setting \cite{BerensteinChangLi,Wolsson}. We will see how generalizations of these results are easy consequences of the results in the previous section.

Let $\alpha = (\alpha_{1} , \ldots , \alpha_{m}) \in  \m N^m,$ and $| \alpha |= \sum \alpha_{i}$. Fix the (multi-index) notation $\delta^\alpha = \delta_1 ^ {\alpha_{1}} \ldots \delta_m ^{ \alpha _{m} }$, 
and $A = (\alpha^{(0)}, \alpha^{(1)} , \ldots , \alpha ^{(n)} ) \in (\m N^m)^{n+1}$. We call  $$\mc W_A:= \left| \left(  \begin{array}{cccc}
\delta ^ { \alpha^{(0)}} y_0 & \delta ^ { \alpha^{(0)}} y_1 & \ldots & \delta ^ { \alpha^{(0)}} y_n  \\
\delta ^ { \alpha^{(1)}} y_0 & \delta ^ { \alpha^{(1)}} y_1 & \ldots & \delta ^ { \alpha^{(1)}} y_n  \\
\vdots & \vdots  & \ddots & \vdots \\
\delta ^ { \alpha^{(n)}} y_0 & \delta ^ { \alpha^{(n)}} y_1 & \ldots & \delta ^ { \alpha^{(n)}} y_n  
 \end{array} \right)   \right|$$ the \emph{Wronskian associated to} $A$.

\begin{thm}\label{linwronski} If $V= \m P^n ( \mc U^\D)$, the constant points of $\m P^n$, then the projective differential variety $ld (V)$ is equal to the zero set in $\m P^n$ of the collection of generalized Wronskians (i.e., as the tuple $A$ varies in $(\m N^m)^{n+1}$). 
\end{thm}

\begin{rem}
This theorem of the Wronskian is well known, we refer the reader to \cite[Chapter II, \S 1, Theorem 1]{KolchinDAAG} for a standard proof. Since we are not yet restricting ourselves to a subcollection of generalized Wronskians, we present a more direct proof (using the language of this paper).
\end{rem}

\begin{proof} 
The vanishing of the collection of generalized Wronskians is clearly a necessary condition for linear dependence. We now show that it is sufficient. Let $(f_0,f_1,\ldots,f_n)$ be a tuple in the zero set of the collection of generalized Wronskians. Consider the matrix
$$M=(\delta^{\alpha}f_i)_{i\leq n, \alpha\in \m N^m}.$$
By our assumption, $M$ has rank at most $n$ and so $V:=\operatorname{ker}M\subseteq \mc U^{n+1}$ is a positive dimensional subspace. We now check that $V$ is stable under the derivations. Let $\delta\in \Delta$ and $v=(v_0,\ldots,v_n)\in V$, then for any $\alpha\in \m N^m$ we have
$$\sum_{i=0}^n \delta^\alpha f_i\cdot \delta v_i=\delta\left(\sum_i \delta^\alpha f_i\cdot v_i\right)-\sum_i\delta(\delta^\alpha f_i)\cdot v_i=0.$$
So $\delta v\in V$. Thus, $(V,\Delta)$ is a $\Delta$-module over $\mc U$. It is well known that a finite dimensional $\Delta$-module over a differentially closed field has a basis consisting of $\Delta$-constants (this can be deduced from the existence of fundamental matrix of solutions of integrable linear differential equations). Hence, there is $(c_0,\ldots,c_n)\in V\cap \mc U^{\Delta}$, and so, in particular, $\sum_i c_i f_i=0$.
\end{proof}

The following corollaries show why our results are essentially generalizations of \cite[Theorem 2.1]{BerensteinChangLi}, \cite[Theorem 2.1]{Walker} and \cite{Wolsson}: 

\begin{corr} Let $k$ be a differential subfield of $K$ and $f_0 , \ldots , f_n\in K$. If $\m P^n(\mc U^\D)$ is $k$-large with respect to $K$, then  $f_0, \ldots , f_n$ are linearly dependent over $\m P^n(k^\D)$ if and only if the collection of generalized Wronskians vanish on $(f_0,\dots,f_n)$. 
\end{corr} 

Since $\m P^n(\mc U^\D)$ is $\m C$-large with respect to any field of meromorphic functions on some domain of $\m C^m$, we have

\begin{corr} The vanishing of the collection of generalized Wronskians is a necessary and sufficient condition for the linear dependence (over $\m C$) of a finite collection of meromorphic functions on some domain of $\m C^m$. 
\end{corr}

In \cite{Walker}, Walker proved that the vanishing of the collection of generalized Wronskians is equivalent to the vanishing of the subcollection of those Wronskians associated to Young-like sets (and this is the least subcollection of Wronskians with this property). A \emph{Young-like} set $A$ is an element of $(\m N^m)^{n+1}$ with the property that if $\alpha \in A$ and $\beta \in \m N^m$ are such that $\beta <\alpha$ in the product order of $\m N^m$, then $\beta \in A$. (When $m=2$, Young-like sets correspond to Young diagrams.)

There are computational advantages to working with Young-like sets, since the full set of generalized Wronskians grows much faster in $(m,n)$ (where $m$ is the number of derivations and $n$ is the number of functions). Even for small values of $(m,n)$ the difference is appreciable, see for example \cite{Walker} for specifics on the growth of the collection of Young-like sets. 

Using Walker's result, we obtain the following corollary:
\begin{corr}
If $V= \m P^n ( \mc U^\D)$, the constant points of $\m P^n$, then the projective differential variety $ld (V)$ is equal to the zero set in $\m P^n$ of the collection of generalized Wronskians associated to Young-like sets. Moreover, this subcollection of Wronskians is the smallest one with this property.
\end{corr}

\section{The catenary problem and related results}  \label{catsection}

In this section we discuss some conjectural questions around completeness and Kolchin's catenary problem. We also take the opportunity to pose stronger forms of the catenary conjecture for algebraic varieties that seem interesting in their own right. We begin by recalling the \emph{catenary problem}:

\begin{prob} Given an irreducible differential variety $V$ of dimension $d>0$ and an arbitrary point $p \in V$, does there exist a long gap chain beginning at $p$ and ending at $V$? By a long gap chain we mean a chain of irreducible differential subvarieties of length $\omega^m \cdot d$. The positive answer to this question is called the Kolchin catenary conjecture. 
\end{prob} 

Under stronger assumptions, various authors have established the existence of long gap chains.  When $p \in V$ satisfies certain nonsingularity assumptions, Johnson \cite{JohnsonCat} established the existence of the chain of subvarieties of $V$ starting at $p$. In \cite{Rosenfeld}, Rosenfeld proves a slightly weaker statement (also with nonsingularity assumptions) and expresses the opinion that the nonsingularity hypotheses might not be necessary; however, except for special cases, the hypotheses have not been eliminated. See \cite[pages 607-608]{BuiumCassidyKolchin} for additional details on the history of this problem.

In a different direction, Pong \cite{PongCat} answered the problem affirmatively, assuming that $V$ is an algebraic variety, but assuming nothing about the point $p$. Pong's proof invokes resolution of singularities (the ``nonsingularity" assumptions of \cite{Rosenfeld, JohnsonCat} are not equivalent to the classical notion of a nonsingular point on an algebraic variety). It is worth mentioning that even though Pong works in the ordinary case, $\Delta=\{\delta\}$, his approach and results readily generalize to the partial case. 

We also have the following weaker form of the catenary conjecture:

\begin{conj}[Weak Catenary Conjecture] \label{weakcat} 
For every positive dimensional irreducible differential variety $V \subseteq \m A^n$ and every zero-dimensional differential subvariety $W \subseteq V$, there is a proper irreducible differential subvariety $V_1$ of $V$ such that $V_1 \cap W \neq \emptyset$ and $V_1 \not \subseteq W$. 
\end{conj} 

This conjecture (although it seems rather innocuous, a proof is not known) is a very easy consequence of the catenary conjecture. Indeed, pick $p \in W$ and pick a long gap chain starting at $p$. Then, since the Kolchin polynomials of the sets in the chain are not equal to each other, at some level the irreducible closed sets in the chain can not be contained in the set $W$. 

In Section \ref{compbound} we suggested the following

\begin{conj}\label{conbound}
A $\Delta$-closed $V \subseteq \m P^n$ is complete if and only if $\pi_2: V \times Y \rightarrow Y$ is a $\Delta$-closed map for every quasiprojective zero-dimensional differential variety $Y.$  
\end{conj}

Even though we are not able to prove this, we show that it is a consequence of the weak catenary conjecture.

\begin{lem} 
The Weak Catenary Conjecture implies Conjecture \ref{conbound}.
\end{lem} 
\begin{proof} 
Towards a contradiction, suppose that $\pi_2:V\times Y\to Y$ is a closed map for every zero-dimensional $Y$, but $V$ is not complete. Then, by Proposition \ref{mainthm}, there must be some $Y$, a positive dimensional irreducible affine differential variety, and an irreducible closed set $X \subseteq V \times Y_1$ such that $\pi_2 (X)$ is not closed, $\pi_2(X)^{cl}$ is positive dimensional, and $W:=\pi_2(X)^{cl}\setminus \pi_2(X)$ is zero-dimensional. We will obtain the desired contradiction by finding a zero-dimensional $Y'$ which witnesses incompleteness. 

Applying (iterating rather) the weak catenary conjecture, we obtain a zero-dimensional and proper irreducible subvariety $Y'$ of $\pi_2(X)^{cl}$ such that  $Y' \cap W \neq \emptyset$ and $Y' \not \subseteq W$. We claim that this $Y'$ witnesses incompleteness. Let $X'=X\cap(V\times Y')$. Then we claim that $\pi_2(X')=\pi_2(X)\cap Y'$ is not closed. Suppose it is, then it would be a proper closed subset of $Y'$. The fact that $W\cap Y'$ is also a proper closed subset of $Y$, would contradict irreduciblity of $Y'$. Thus, $\pi_2(X')$ is not closed, and the result follows.
\end{proof}

\begin{rem}  Note that restricting to specific families of differential varieties $V$ does not necessarily restrict the varieties on which one applies the weak catenary conjecture in the proof of the lemma. Thus, for our method of proof, Conjecture \ref{weakcat} is (a priori) used in full generality.
\end{rem}

\subsection{A stronger form of the catenary problem for algebraic varieties}

In this section we describe some of the difficulties of the catenary problem for algebraic varieties. As we mentioned already, this case follows from results of Pong \cite{PongCat}; however, in what follows we propose a different approach to this problem.

A differential ring is called a \emph{Keigher ring} if the radical of every differential ideal is again a differential ideal. The rings we will be considering will be assumed to be Keigher rings. Note that every Ritt algebra is a Keigher ring (see for instance \cite[\S1]{MMP}).
 
Given $f: A \rightarrow B$ a differential homomorphism of Keigher rings, we have an induced map $f^*: Spec \, B \rightarrow Spec \, A$ given by $f^* ( \mf p )  = f^{-1} ( \mf p )$. We denote by $f^*_ \Delta : Spec ^ \Delta B \rightarrow Spec ^ \Delta A$ the restriction of the map $f^*$ to the differential spectrum. We have the following differential analogs of the going-up and going-down properties:

\begin{defn}\label{goingupdown} 
Suppose we are given some chain $\mf p_1 \subseteq \ldots \subseteq \mf p_n$ with $\mf p_i \in Spec^\Delta \, f(A)$ and any $\mf q_1 \subseteq \ldots \subseteq \mf q _m  \in Spec^\Delta \, B$ such that for each $i \leq m,$ $ \mf q_i \cap f(A) =\mf p_i$. We say that $f$ has the \emph{going-up property for differential ideals} if given any such chains $\mf p$ and $\mf q$, we may extend the second chain to $\mf q_1 \subseteq \ldots \subseteq \mf q_n$ where $\mf q_i\in Spec^\Delta B$ such that for each $i \leq n,$ $\mf q_i \cap f(A) = \mf p_i$. One can analogously define when $f$ has the \emph{going-down property for differential ideals}.
\end{defn}

When $(A, \Delta) \subseteq (B,\Delta) $ are integral domains, $B$ is integral over $A$, and $A$ is integrally closed, then the differential embedding $A \subseteq B$ has the going-down property for differential ideals. Dropping the integrally closed requirement on $A$, one can still prove the going-up property for differential ideals \cite[Proposition 1.1]{PongCat}. In what follows we will see how these results are consequences of their classical counterparts in commutative algebra.

Let us review some developments of differential algebra which are proved in \cite{Trushin}. We will prove the results which we need here in order to keep the exposition self-contained and tailored to our needs. Let $f: A \rightarrow B$ be a differential homomorphism of Keigher rings. The fundamental idea, which Trushin calls \emph{inheritance}, is to consider one property of such a map $f$ considered only as a map of rings and another property of $f$ as a map of differential rings and prove that the properties are equivalent. As we will see, in certain cases, one can reduce the task of proving various differential algebraic facts to proving corresponding algebraic facts. 

\begin{lem} \label{one}
Let $\mf p \subset A$ be a prime differential ideal. The following are equivalent: 
\begin{enumerate} 
\item $\mf p = f^{-1} (f  (\mf p ) B)$,
\item $(f^*)^{-1} (\mf p ) \neq \emptyset$,
\item $(f_\Delta ^*)^{-1} ( \mf p ) \neq \emptyset$.
\end{enumerate} 
\end{lem}
\begin{proof} $(1) \Leftrightarrow (2)$ is precisely \cite[Proposition 3.16]{AtiyahMac}. $(3) \Rightarrow (2)$ is trivial. To show that $(2) \Rightarrow (3),$ note that $(f^*)^{-1} (\mf p)$ is homeomorphic to $Spec \, B_ \mf p / \mf p B _ \mf p$. The fact that the fiber is nonempty means that $B_ \mf p / \mf p B _ \mf p$ is not the zero ring. Since it is a Keigher ring, $Spec^\Delta \, B_ \mf p / \mf p B _ \mf p$ is nonempty (see \cite{differentialschemes}) and naturally homeomorphic to $(f^*_ \Delta)^{-1}  ( \mf p)$. 
\end{proof}

The following results are easy applications of the lemma:

\begin{corr} \label{twoone} \
\begin{enumerate}
\item [(i)] If $f^*$ is surjective, so is $f^*_\Delta$. 
\item [(ii)] If $f$ has the going-up property, then $f$ has the going-up property for differential ideals. 
\item [(iii)] If $f$ has the going-down property, then $f$ has the going-down property for differential ideals. 
\end{enumerate}
\end{corr}

Of course, by applying the previous corrollary to integral extensions with the standard additional hypotheses, we get the desired analogs of the classical going-up and going-down properties (see \cite{PongCat} or \cite{Trushin}, where the results were reproved). 

\begin{prop}\label{downdownbaby} Suppose that $A$ is a Ritt algebra, $Spec^\Delta \, A$ is Noetherian, $B$ is a finitely generated differential ring over $A$,  and the map $f:A \rightarrow B$ is the embedding map. Then the following are equivalent. 
\begin{enumerate} 
\item $f$ has the going-down property for differential ideals,
\item $f_\Delta^*$ is an open map (with respect to the $\Delta$-topology).
\end{enumerate}
\end{prop} 
\begin{proof}
Let us prove that (2) implies (1). Let $ \mf q \in Spec ^ \Delta \, B$ and let $\mf p = f^{-1} ( \mf q).$ Since we are interested in differential ideals contained in $\mf q$, it will be useful to consider the local ring $B_ \mf q$, and we note that $B_ \mf q = \varinjlim _{t \in B \backslash \mf q } B_t$. 

By \cite[Exercise 26 of Chapter 3]{AtiyahMac}, $f^* (Spec \, B_ \mf q) = \bigcap _{t \in B \backslash \mf q} f  ^* ( Spec  \, B_t)$. Now, by Corollary \ref{twoone}, surjectivity of $f^*$ implies surjectivity of $f_\Delta^*$, so 
$$f^* _\Delta( Spec ^ \Delta \, B _ \mf q ) = \bigcap _{t\in B \backslash \mf q} f_ \Delta ^* (Spec^ \Delta ( B_t )).$$ Since $f_\Delta ^*$ is an open map, $f^* _\Delta ( Spec ^ \Delta \, B_t )$ is an open neighborhood of $ \mf p$ and so it contains $Spec ^ \Delta  \,  A_ \mf p$. 

We have proved that, for any $ \mf q \in Spec ^ \Delta \, B$, the induced map $f_\Delta ^* :Spec ^ \Delta \, B_ \mf q \rightarrow Spec ^ \Delta A_ \mf p$ is a surjective map, where $\mf p = f^{-1} ( \mf q)$. Since differential ideals contained in $ \mf p$ correspond to differential ideals in $A_ \mf p$, we have established the going-down property for differential ideals. 

Now we prove that (1) implies (2). Take $\mf p \in  f_ \Delta ^* (Spec ^\Delta \, B_t )$ with $f_\Delta ^* ( \mf q) =\mf p$. Take some irreducible closed subset $Z \subseteq Spec ^ \Delta \, A$ for which $Z \cap  f_ \Delta ^* (Spec ^\Delta \, B_t )$ is nonempty. Now take some $\mf p_1 \in Spec ^ \Delta \, A$ with $\mf p_1 \subset \mf p$. By the going-down property for differential ideals, $ \mf p_1 = f^*_\Delta (\mf q_1)$ for some $ \mf q_1 \in Spec^ \Delta \, B$. Noting that $\mf q_1$ is in $Spec ^\Delta \, B_t$, we see that $ f_ \Delta ^* (Spec ^\Delta \, B_t ) \cap Z$ is dense in $Z$. Since the set $ f_ \Delta ^* (Spec ^\Delta \, B_t ) $ is constructible \cite[Statement 11]{Trushin} and thus contains an open subset of its closure, $Z \cap f_ \Delta ^* (Spec ^\Delta \, B_t ) $ contains an open subset of $Z$. This holds for arbitrary $Z$, so $f_ \Delta ^* (Spec ^\Delta \, B_t ) $ is open. 
\end{proof}

In many circumstances the existance of long gap chains can be deduced from the existance of such chains in affine space (which are well known):

\begin{fact} \label{affinechain} Let $p \in \m A^d $ be an arbitrary point. Then there is a long gap chain defined over $\mathbb Q$ starting at $p$ and ending at $\m A^d$.
\end{fact} 

The following is a specific example which establishes the previous fact. Moreover, this example produces a family of differential algebraic subgroups of the additive group such that for every $\alpha<\omega^m$ there is an element in the family whose Lascar rank is ``close'' to $\alpha$.

\begin{exam}\label{Affinepoint}
We produce a family $\{G_r: r\in n\times\NN^m\}$ of differential algebraic subgroups of the additive group $(\mathbb{A}^n,+)$ with the following properties. For every $r=(i,r_1,\dots,r_m)\in n\times \NN^m$,
$$\w^m i+\sum_{j=k}^m\w^{m-j}r_j \leq U(G_r) < \w^m i+\w^{m-k} (r_k+1)$$
where $k$ is the smallest such that $r_k>0$, and if $r,s\in n\times \NN^m$ are such that $r< s$, in the lexicographical order, then the containment $G_r\subset G_s$ is strict. Here $U(G_r)$ denotes the Lascar rank of $G_r$. We refer the reader to \cite{MMP} for definitions and basic properties of this model-theoretic rank.

For $r=(i,r_1,\dots,r_m)\in n\times \NN^m$, let $G_r$ be defined by the homogeneous system of linear differential equations in the $\D$-indeterminates $x_0,\dots,x_{n-1}$, 
\begin{displaymath}
\left\{
\begin{array}{c}
\d_1^{r_1+1}x_i=0,  \\
\d_2^{r_2+1}\d_1^{r_1}x_i=0,\\
\vdots\\
\d_{m-1}^{r_{m-1}+1}\d_{m-2}^{r_{m_2}}\cdots\d_1^{r_1}x_i=0, \\ \d_{m}^{r_m}\d_{m-1}^{r_{m-1}}\cdots\d_1^{r_1}x_i=0,
\end{array}
\right.
\end{displaymath}
together with 
$$x_{i+1}=0,\cdots, x_{n-1}=0.$$
Note that if $r,s\in n\times\NN^m$ are such that $r< s$, then $G_r\subset G_s$ is strict. We first show that $$U(G_r)\geq \w^m i +\sum_{j=1}^m\w^{m-j}r_j.$$ 
We prove this by transfinite induction on $r=(i,r_1, . . . , r_m)$ in the lexicographical order. The base case holds trivially. Suppose first that $r_m \neq 0$ (i.e., the succesor ordinal case). Consider the (definable) group homomorphism $f:(G_r,+)\to (G_r,+)$ given by $f(x_i)=\d_m^{r_m -1}x_i$. Then the generic type of the generic fibre of $f$ is a forking extension of the generic type of $G_r$. Since $f$ is a definable group homomorphism, the Lascar rank of the generic fibre is the same as the Lascar rank of $\operatorname{Ker}(f)=G_{r'}$, where $r'=(i,r_1,\dots,r_{m-1},r_m -1)$. By induction, 
$$U(G_{r'})\geq \w^m i+\sum_{j=1}^{m-1}\w^{m-j}r_j +(r_m-1).
$$ Hence, $$U(G_{r})\geq \w^m i+ \sum_{j=1}^m\w^{m-j}r_j.$$ 

Now suppose $r_m = 0$ (i.e., the limit ordinal case). Suppose there is $k$ such that $r_k \neq 0$ and that $k$ is the largest such. Let $\ell\in \w$ and $r' = (i,r_1,\dots,r_k -1,\ell,0,...,0)$. Then $G_{r'}\subset G_r$ and, by induction, 
$$U(G_{r'})\geq \w^m i+\sum_{j=1}^{k-1}\w^{m-j}r_j+\w^{m-k} (r_{k}-1)+\w^{m-k-1}\ell.$$ 
Since $\ell$ was arbitrary, $$U(G_r)\geq \w^m i+\sum_{j=1}^{k}\w^{m-j}r_j.$$ 
Finally suppose that all the $r_k$'s are zero and that $i>0$. Let $\ell\in \w$ and $r'=(i-1,\ell,0,\dots,0)$. Then again $G_{r'}\subseteq G_r$ and, by induction, 
$$U(G_{r'})\geq \w^m (i-1)+\w^{m-1}\ell.$$ 
Since $\ell$ was arbitrary, $$U(G_r)\geq \w^m i.$$ 
This completes the induction.

Let $k$ be the smallest such that $r_k>0$ and let $tp(a_0,\dots,a_{\ell-1}/K)$ be the generic type of $G_r$. We now show that if $i=0$, then $\D$-type$(G_r)=m-k$ and $\D$-dim$(G_r)=r_k$ (here $\Delta$-dim denotes the typical $\Delta$-dimension, see Chapter II of \cite{KolchinDAAG}). As $i=r_1=\cdots=r_{k-1}=0$, we have $a_1=\cdots=a_{n-1}=0$ and $\d_1 a_0=0,\,\dots,\,\d_{k-1} a_0=0,\, \d_k^{r_k+1}a_0=0$ and $\d_k^{r_k}a_0$ is $\D_{k}$-algebraic over $K$ where $\D_k=\{\d_{k+1},\dots,\d_m\}$. It suffices to show that $a_0,\d_k a_0,\dots, \d_k^{r_k-1}a_0$ are $\D_{k}$-algebraically independent over $K$. Let $f$ be a nonzero $\Delta_k$-polynomial over $K$ in the variables $x_0,\ldots,x_{r_k-1}$, and let $g(x)=f(x,\d_k x,\dots,\d_k^{r_k-1} x)\in K\{x\}$. Then $g$ is a nonzero $\D$-polynomial over $K$ reduced with respect to the defining $\D$-ideal of $G_r$ over $K$. Thus, as $a$ is a generic point of $G_r$ over $K$, $$0\neq g(a)=f(a,\d_k a,\dots,\d_k^{r_k-1} a),$$ as desired. Applying this, together with McGrail'
s \cite{McGrail} upper bounds for 
Lascar rank, we get $$U(G_r)<\w^{m-k}(r_k+1).$$ 
For arbitrary $i$, the above results show that $U(a_j/K)=\w^m$ for $j<i$, and $U(a_i/K)< \w^{m-k}(r_k+1)$ and $U(a_j/K)=0$ for $j>i$. Applying Lascar's inequality we get:
$$U(G_r)\leq U(a_0/K)\oplus\cdots\oplus U(a_{n-1}/K)<\w^m i+\w^{m-k}(r_k+1),$$
where $\oplus$ denotes the Cantor sum of ordinals. This proves the other inequality.
\end{exam}

\begin{rem}\label{remP} We now echo some remarks from \cite{PongCat}: 
\begin{enumerate} \item The catenary problem is essentially local; i.e, if you can find an ascending chain of irreducible differential subvarieties of an open set, then taking their closures results in an ascending chain in the given irreducible differential variety. 
\item As Pong \cite[page 759]{PongCat} points out, truncated versions of the differential coordinate rings of singular algebraic varieties do not satisfy the hypotheses of the classical going-down or going-up theorems with respect the ring embedding given by Noether normalization. (In \cite{PongCat}, this difficulty is avoided using resolution of singularities.)
\end{enumerate}
\end{rem} 

In light of Theorem \ref{downdownbaby}, Fact \ref{affinechain} and Remark \ref{remP}(1), one can see that the following question is a stronger version of the Kolchin catenary problem for algebraic varieties:

\begin{ques}\label{qeso}
Let $f: V \rightarrow \m A^d$, where $d=dim\, V$, be a finite open map of irreducible affine algebraic varieties. Then, if $f_\Delta$ denotes $f$ when regarded as a map of differential algebraic varieties, is $f_\Delta$ an open map?
\end{ques} 

In the next question, we will use the terminology of \cite{MSHasse}. When $V$ is an algebraic variety, we let $V_ \infty$ denote the inverse image of the prolongation spaces $\varprojlim _n \tau _ n (V)$ with respect to the appropriate finite free algebra corresponding to $\Delta$ \cite[Example 2.4]{MSHasse}. When $f: V \rightarrow W$ is a map of varieties, there is a naturally induced map $f_ \infty : V _\infty \rightarrow W _\infty$. 
The following question would yield, by quantifier elimination and Theorem \ref{downdownbaby}, a positive answer to Question \ref{qeso}: 

\begin{ques} Let $f: V \rightarrow \m A^d$, where $d=dim\, V$, be a finite open map of irreducible affine algebraic varieties. Let $f_\infty : V_\infty \rightarrow \m A^d _\infty$ be the induced map on their prolongation spaces. Is $f_\infty$ an open map? 
\end{ques} 

We do not know the answers to either of these questions in general; however, there is some evidence for the first one. We observe that at least, in the context of Question \ref{qeso}, the image of every $\Delta$-open set contains a $\Delta$-open set. Let $f$ and $f_\Delta$ be as in Question \ref{qeso}, and let $U$ be a $\Delta$-open subset of $V$. Then, the Lascar rank of $U$ is $\omega^m\cdot d$ (where $d=dim\, V$). It follows, from Lascar inequalities and the fact that $f_\Delta$ has finite fibres, that $f_\Delta(U)$ has Lascar rank $\omega^m\cdot d$. By quantifier elimination, $f_\Delta(U)$ is constructible and so it must contain a $\Delta$-open set. 

The second question does not appear to be answered in the literature on arc spaces, which is pertinent under the additional assumption that the variety $V$ is defined over the constants. The question can not be obviously answered by restricting ones attention to the finite level prolongation spaces even in the ordinary differential case for varieties defined over the constants (i.e., the case of arc spaces used in algebraic geometry). For instance, let $C$ be a curve and take $f: C \rightarrow \m A^1$ to be the finite open map given by Noether normalization. The induced map of tangent bundles of a cuspidal curve $Tf: TC \rightarrow \m A^2$ is not an open map; the extra component of the tangent bundle over the cusp gets mapped to a single point in $\m A^2$. 

\begin{rem}
Pong's solution \cite{PongCat} to the Kolchin catenary problem for algebraic varieties avoids the stronger forms we have given here. Instead of asking about the general going-down property (for differential ideals) for the map coming from Noether normalization, Pong uses resolution of singularities to reduce the question to smooth varieties. 
\end{rem}

\bibliography{research}{}
\bibliographystyle{plain}
\end{document}